\documentclass{article}
\usepackage{amsthm}
\usepackage{amssymb}
\usepackage{amsmath}
\usepackage{float}
 \usepackage{mathtools}
\usepackage{comment}

\theoremstyle{plain}
\newtheorem{theorem}{Theorem}[section]
\newtheorem{corollary}[theorem]{Corollary}
\newtheorem{lemma}[theorem]{Lemma}

\theoremstyle{definition}
\newtheorem{definition}[theorem]{Definition}

\theoremstyle{remark}
\newtheorem{remark} [theorem]{Remark}

\renewcommand{\L}{\mathcal{L}}

\newcommand{\Ii}{\mathcal{I}}

\newcommand{\R}{\mathbb{R}}

\newcommand{\Z}{\mathbb{Z}}
\newcommand{\C}{\mathbb{C}}

\newcommand{\lieg}{\mathfrak{g}}

\newcommand{\liep}{\mathfrak{p}}
\newcommand{\liea}{\mathfrak{a}}

\newcommand{\lieh}{\mathfrak{h}}
\newcommand{\liek}{\mathfrak{k}}

\newcommand{\N}{\mathbb{N}}

\newcommand{\llbracket}{[\![}
\newcommand{\rrbracket}{]\!]}

\title{Generalized k-contact Structures}
\author{Uir\'a N. Matos de Almeida
\thanks{The Author thanks CAPES ans CNPq for the financial support, an anonymous referee who pointed a flaw in my earlier proof and for his general insights and corrections and, lastly, my advisors Thierry Barbot and Carlos Maquera.}}

\begin{document}

\maketitle

\begin{abstract}
With the goal to study and better understand algebraic Anosov actions of $\R^k$, we develop a higher codimensional analogue of the contact distribution on odd dimensional manifolds,  call such structure a generalized $k$-contact structure. We show that there exist an $\R^k$-action associated with this structure, afterwards, we relate this structure with the Weyl chamber actions and a few more general algebraic Anosov actions, proving that such actions admits a compatible generalized $k$-contact structure.

\end{abstract}

\tableofcontents 

\section{Introduction}

Our main interest is the study of Anosov systems, in particular, Anosov actions of $\R^k$. There are two possibilities:  either the subbundle $E^+\oplus E^-$ is integrable or not. In the first case T. Barbot and C. Maquera  \cite{Ba-Maq2} proved that the action is in fact a suspension of a $\Z^k$ action. We are interested in the second case, more precisely, when it is "maximally non integrable", when $k=1$ this condition means that $E^+\oplus E^-$ is a contact distribution, that is, the flow is a contact Anosov flow.

Motivated by our interest in Anosov actions of $\R^k$, we propose a  definition of contact structures associated with distributions of higher co-dimension , which we called generalized $k$-contact structure. This geometric structure comes with a naturally defined $\R^k$-action  which preserves this structure, which we call a contact action, and we are lead to define the notion of contact Anosov actions. On the special case of $k=1$, this geometric structure is actually a contact structure and the notion of contact Anosov action reduces to the usual contact Anosov flow.

It is well known that geodesic flows in general are contact flows, moreover, on a negatively curved manifold, the geodesic flow is Anosov \cite{Anosov}. In 1992 Y. Benoist, P. Foulon and F. Labourie \cite{BFL} proved, under some additional hypothesis, the converse. To be precise, they proved that contact Anosov flows with smooth invariant bundles are geodesic flows on a manifold of constant negative curvature. In other words, contact Anosov flows with smooth invariant bundles can be seen as the action of a one parameter subgroup on a locally homogeneous space $\Gamma\backslash G\slash H$ where $G$ is a rank 1 semisimple Lie group. Such flows are particular cases of what is known as Weyl chamber actions, which turns out to be the main example of Anosov actions of $\R^k$. It is natural, therefore, to ask:
\begin{enumerate}
    \item {} For general $k$, are contact Anosov actions (with smooth invariant bundles) Weyl chamber actions?
    \item{} For general $k$, are Weyl chamber actions contact Anosov actions?
\end{enumerate}

The first question was answered positively in \cite{Almeida} and is related to a standing conjecture by B. kalinin and R. Spatzier \cite{kal} stating the algebricity of Anosov actions of $\R^k$ for $k\geq2$. The goal of this paper is to answer the second question. We relate the generalized $k$-contact structures to some other definitions found in the literature, and show that Weyl chamber actions on semi-simple Lie groups are in fact the contact action of an associated generalized k-contact structure. 

Like the usual contact case, the subbundle $E^+\oplus E^-$ is not integrable, in fact,  our action does not admits a globally transverse submanifold (Lemma \ref{transverse}), which implies that the action can't be a suspension. 

In particular, following the classification of algebraic Anosov actions (of nilpotent Lie groups) given by T. Barbot and C. Maquera \cite{Ba-Maq3}, we obtain that every algebraic Anosov action of $\R^k$ that is not a suspension comes from a generalized $k$-contact structure. Our main theorem is thus:
\begin{theorem}\label{TheoremA}
 Let $(G,K,\Gamma,\mathfrak a)$ be a Weyl chamber action. Then there exists an generalized $k$-contact structure on $\Gamma\backslash G\slash K$ such that the induced contact action is Anosov and it coincides with the Weyl chamber action.
\end{theorem}

As a corollary we obtain:
\begin{theorem}\label{TheoremB}
Let $(G,K,\Gamma,\mathfrak a)$ be an algebraic Anosov action which is not a suspension. Then there exists an generalized $k$-contact structure on $\Gamma\backslash G\slash K$ such that the induced contact action is Anosov and it coincides with the algebraic action.
\end{theorem}

This paper is organized in the following way:

In section 2, we give our proposed definition of generalized $k$-contact structures, and it's associated contact action. We prove some general results and give some basic examples. In particular, we show how a generalized $k$-contact structure on a given manifold $B$ can induce a generalized $k+l$-contact structure on a principal $\mathbb T^l$-bundle $M\to B$  over $B$ with a flat connection. 

In section 3, we relate the generalized k-contact structure with Anosov $\R^k$  actions to define a contact Anosov action.

In section 4, we illustrate, via an example, and prove that this construction, in fact, work for general Weyl chamber actions, proving Theorem \ref{TheoremA}.

In section 5, we recall the classification of algebraic Anosov actions of nilpotent Lie groups, given by T. Barbot and C. Maquera \cite{Ba-Maq3}, specifying the case where the Lie group is in fact Abelian. We use this classification to show Theorem \ref{TheoremB}


\section{Generalized $k$-contact structures}\label{subsecgenkcont}
\begin{definition}\label{genkcont}
Let $M$ be a smooth manifold (not necessarily compact) of dimension $2n+k$.
	A generalized $k$-contact on $M$ is a collection of smooth $1$-forms $\{\alpha_1,\dots,\alpha_k\}$, which are pointwise linearly independent, and a $C^{\infty}$-splitting $TM = I\oplus F$, 	$\dim I = k$,  such that, for every $1\leq j\leq k$ we have, 
	\begin{enumerate}
		\item {}$F = \bigcap_{i=1}^k\ker\alpha_i$
		\item {}$\ker(d\alpha_j) = I$
	\end{enumerate}
	We denote this structure by $(M,\alpha,TM = I\oplus F)$.
\end{definition}
\begin{remark}\label{volform} 
    Notice that condition (2) is actually equivalent to
    \begin{itemize}
         \item [(3)]($d\alpha_j)_{|_F}$ is non degenerate and $I\subset\ker (d\alpha_j)$
    \end{itemize}
	moreover, as $dim(F)=2n$, the non degeneracy of $(d\alpha_j)_{|_F}$ is equivalent to 
	$$(d\alpha_j^n)_{|_F}\neq 0$$
	it follows that, for every $j = 1,\dots,k$ $$\alpha_1\wedge\cdots\wedge\alpha_k\wedge d\alpha_j^n\enskip\text {is a volume form}.$$
\end{remark}
\begin{lemma}
	The $1$-forms that define a generalized $k$-contact structure have constant rank $2n+1$
	\footnote{Remember that the rank of a differential form $\omega$ of rank $p$ at a point $y$ is the co-dimension of the characteristic space $$\mathcal C(\omega)(y) = \{X\in T_pM\;;\; i_X\omega(y) =0\;\;and\;\; i_Xd\omega(y) = 0\}$$}.
		
\end{lemma}	
\begin{proof}
	As $dim I = k$  we have $dim(F) = 2n$.  As $\ker(d\alpha_j) = I$ and $TM = I\oplus F$ then, $(d\alpha_j^{n+1})_p = 0$ for every point $p\in M$, moreover, as we already remarked,  $(d\alpha_j^n)_p\neq 0$. Thus, $rank(\alpha_j)$ on each point is either $2n$ or $2n+1$
	Finally, as ${\alpha_j}_{|_{\ker d\alpha_j}}\neq 0$ then $\alpha_j$ has odd rank everywhere, that is, $rank(\alpha_j)=2n+1$ at every point.
\end{proof}

\begin{lemma}
	The distribution $F$ in a generalized $k$-contact structure is non integrable.
\end{lemma}
\begin{proof}
	Suppose that $F$ is integrable. From Frobenius Theorem, this is equivalent to $[F,F]\subset F$. Take two vector fields $Z,W\in\Gamma(M,F)$, in particular, we have $$\alpha_j(Z) = \alpha_j(W)= \alpha_j([Z,W])=0$$ and thus
	$$d\alpha_j(Z,W) = Z(\alpha_j(W)) - W(\alpha_j(Z))-\alpha_j([Z,W])=0$$
	which contradicts the fact that $d\alpha_j$ is non degenerate on $F$.
\end{proof}
\begin{remark}
	The distribution $F$ is maximally non integrable in the following sense: For every vector field $Z$ tangent o $F$ there exists a vector field $W$ also tangent to $F$ such that $[Z,W]$ is not tangent to $F$.
\end{remark}

\begin{lemma}\label{reebfields}
	For each $j$, there is a unique vector field $X_j\in \Gamma(M,I)$ such that $\alpha_i(X_j) = \delta_{ij}$. These vector fields are called Reeb vector fields. Moreover, the Reeb vector fields commute one with each other:
	$$[X_i,X_j]=0$$
\end{lemma}
\begin{proof}
    For $p \in M$, the linear functionals $\beta_i(p)$ which are the
restrictions of the $\alpha_i(p)$, $1 \leq i \leq k$, to the k-dimensional vector space $I(p)$ satisfy 
$$\bigcap_i\ker(\beta_i(p)) = \{0\}$$

 Hence, they form a basis of the dual space of $I(p)$. We define $X_i(p)$, $1 \leq i \leq k$, as the dual basis. The smoothness of the forms $\alpha_1,\dots,\alpha_k$ implies the smoothness of the vector fields $X_1,\dots,X_k$.

	
	To verify that they commute, we recall that if $\Omega$ is a volume form, then, $i_Z\Omega = 0$ implies that $Z=0$, thus, we 
	must show that, for any $l$,
	$$i_{[X_i,X_j]}\alpha_1\wedge\dots\wedge\alpha_k\wedge d\alpha_l^n = 0.$$
	
	As $\alpha_1\wedge\dots\wedge\alpha_k\wedge d\alpha_l^n$ is a volume form for any $l$ (Remark \ref{volform}), this will imply that  $[X_i,X_j]=0$.
	
	First we notice that for any $i,l$ we have $d\alpha_i\wedge d\alpha_l^n = 0$, and thus, for any $1\leq j\leq k$ we have
	$$d(\alpha_j\wedge d\alpha_l^n) = 0$$
	and more generally, for any $1\leq j_1<\dots <j_s\leq k$ we have
	$$d(\alpha_{j_1}\wedge\dots\wedge \alpha_{j_s}\wedge d\alpha_l^n) = 0$$
	
	Thus, for any $1\leq j\leq k$
	$$(d\circ i_{X_j})\alpha_1\wedge\dots\wedge\alpha_k\wedge d\alpha_l^n = d(\alpha_1\wedge\dots\wedge \alpha_{j-1}\wedge \alpha_{j+1}\wedge\dots\wedge\alpha_k\wedge d\alpha_l^n)= 0$$
	and similarly
	$$(d\circ i_{X_j}\circ i_{X_i})\alpha_1\wedge\dots\wedge\alpha_k\wedge d\alpha_l^n=0$$
	
	Moreover, as $\alpha_1\wedge\dots\wedge\alpha_k\wedge d\alpha_l^n$ is a volume form, we have
	$$d(\alpha_1\wedge\dots\wedge\alpha_k\wedge d\alpha_l^n)=0$$
	
	Thus, using Cartan's formula\footnote{The Cartan's formula is $\L_X = i_X\circ d + d\circ i_X$ where $\L_X$ denotes the Lie derivative. A classical consequence (Kobayashi, S and Nomizu, K. \cite{kobayashi}, Section I.3 Proposition 3.10) of Cartan's formula is the identity	$$i_{[A,B]} = [i_A,\L_B]$$}: 
	\begin{align*}
	i_{[X_i,X_j]}\alpha_1\wedge\dots\wedge\alpha_k\wedge d\alpha_l^n& = [i_{X_i},\L_{X_j}]\alpha_1\wedge\dots\wedge\alpha_k\wedge d\alpha_l^n\\
	& = [i_{X_i},i_{X_j}\circ d + d\circ i_{X_j}]\alpha_1\wedge\dots\wedge\alpha_k\wedge d\alpha_l^n=0
	\end{align*}
	
\end{proof}	

\begin{definition}
For a given generalized $k$-contact structure the induced $\R^k$ action given by the Reeb vector fields will be called a contact action.
\end{definition}

\begin{remark}
	Let $\phi$ be a contact action. Notice that $I$ is precisely $T\phi$, the distribution tangent to the action. Moreover, if the splitting $TM = I\oplus F$ is smooth, then so are the vector fields $X_j$ and therefore so is the action $\phi$.
	
	It is sometimes convenient to denote a generalized $k$-contact structure on $M$ as the $4$-tuple $(M,\alpha,\phi,F)$, where $\alpha = (\alpha_q,\dots,\alpha_k)$, $\phi$ denotes the contact action and $F\leq TM$ is the $\phi$-invariant subbundle where $d\alpha_j$ is non degenerate.
\end{remark}

\begin{lemma}\label{transverse}
    Let $(M,\alpha,\phi,F)$ be a generalized $k$-contact compact manifold. Then the contact action does not admits a global transverse section, that is, a compact embedded submanifold $N$ of codimension $k$ which is everywhere transverse to the orbits of the action $\phi$.
\end{lemma}
\begin{proof}
    In fact, suppose that $N\subset M$ is a global transverse section for the contact action, then the restriction of $d\lambda_j^n$ to $N$ is a volume form on $N$, but $d\lambda_j^n = d(\lambda_j\wedge d\lambda_j^{n-1})$. By Stokes theorem, $N$ has volume zero which is absurd. 
\end{proof}

\begin{remark}
Among the examples of generalized $k$-contact structures, we have the notion of contact pair, introduced by G. Bande, and A. Hadjar \cite{Bande} and the notion of $r$-contact structure, introduced by P. Bolle in \cite{Bolle}. The Contact pair case is a generalized $2$-contact structure, such that, the $1$-forms $\alpha_1$ and $\alpha_2$ are of the form $\alpha+\eta$ and $\alpha-\eta$, where $\alpha$ and $\eta$ are of constant rank. The $r$-contact structure of P. Bolle is a generalized $r$-contact structure, such that $d\alpha_i = d\alpha_j$ for every $i,j$.

\end{remark}

\begin{lemma}\label{PBundle}
    Let $(B,\alpha,TB = I\oplus F)$ be  a generalized $r$ contact manifold and $\pi: M\to B$ a principal torus bundle (with standard fiber $\mathbb T^l$). Suppose that this bundle admits a flat connection, then, this connection naturally induces a generalized $r+l$ contact structure on $M$.
\end{lemma}

\begin{proof}
    	Let $Y_i$, $i = 1,\dots,l$ be the vector fields that generates the $\mathbb T^l$-action. Consider on each torus, fibers of the bundle, the canonical $1$-forms $\xi_1,\dots,\xi_l$, $\xi_i(Y_j) = \delta_{ij}$. Consider a flat connection on $M$, and let $TM = H \oplus V$ the associated decomposition into horizontal ($H$) and vertical ($V$) bundles. Using this connection we can extend the forms $\xi_1,\dots,\xi_l$ to global forms on $M$.
	
	
	As $\mathbb T^l$ is abelian, it follows that 
	\begin{align}\label{vv}
	    d\xi_i(Y_l,Y_k) = -\xi_i([Y_l,Y_k])=0
	\end{align}
	for any $l,k$. It is known 
	that for any horizontal vector field $Y$, the commutator $[Y_i,Y]$ is horizontal. Therefore, the $1$-forms $\xi_i$ satisfies, for any horizontal vector field $W$,
	\begin{align}\label{aids}
		d\xi_i(Y_l,W) = Y_l(\xi_i(W)) - W(\xi_i(Y_l)) - \xi_i([Y_l,W]) = 0
	\end{align}
	and thus $i_{Y_l}d\xi = 0$. Moreover, as the connection is flat, the horizontal distribution is involutive\footnote{A distribution $D\subset TM$ is involutive if $[Y,Z]$ is tangent to $D$ for any vector fields $Y,Z$ tangent to $D$} and therefore, for any two horizontal vector fields $W_1,W_2$ we have 
	\begin{align}\label{adsna}
d\xi_i(W_1,W_2) = W_1(\xi_i(W_2)) - W_2(\xi_i(W_1) - \xi_i([W_1,W_2])=0
	\end{align}
	
	It follows from (\ref{vv}), (\ref{aids}) and (\ref{adsna}) that $d\xi_i = 0$. In particular, for any $1\leq i\leq l$ and $1\leq j\leq r$, we have $$d(\xi_i + \pi^*\alpha_j) =\pi^*d\alpha_j$$

	As we have the splitting $TB = I\oplus F$, the fibres $H_p$ of the horizontal bundle $H$ can be identified with the fibres $T_{\pi(p)}B$ of the tangent bundle $TB$, we have an induced splitting $H = \hat I\oplus \hat F$. We denote $\hat V = V\oplus \hat I\subset TM$.  
	It is clear that, for any $j$, $\ker (d\pi^*\alpha_j) = \hat V$, and thus, for any $1\leq i\leq l$ and $1\leq j\leq r$, 
	  $\hat V =\ker d(\xi_i + \pi^*\alpha_j).$
	  
	  From this considerations, it is clear that, for any $(j_1,\dots,j_l) \in \{1,\dots,r\}^l$, we have the following generalized $(l+r)$-contact structure on $M$:  
	$$(M,\xi_1 +\pi^*\alpha_{j_1} ,\dots,\xi_l + \pi^*\alpha_{j_l} ,\pi^*\alpha_1,\dots,\pi^*\alpha_r, TM = \hat V\oplus \hat F)$$
	
	In particular, if $B$ is a contact manifold (equivalently, $1$-contact or generalized $1$-contact manifold), and $M\to B$ is a principal torus bundle with a flat connection, then, the induced structure we have constructed above is a $(l+1)$-contact manifold. 
\end{proof}
\begin{remark}
    In the previous proof, we made some arbitrary choices in the construction of the $1$-forms of the induced generalized $k$-contact structure and we may question why is this induced structure "natural". However, while there were some arbitrary choices on the construction of the $1$-forms, the flat connection induces both a natural splitting $TM = \hat V\oplus \hat F$, and a natural lift of the $\R^l$ action to $M$ which commutes with the natural principal bundle action, and thus, it give us a natural $\R^{l+r}$ action.  The following lemma show us that the most important part of the definition of a generalized $k$-contact structure is the splitting and the associated action, which give us some freedom to choose the $1$-forms. This justifies the use of the words "naturally induced generalized contact structure".
\end{remark}

\begin{lemma}\label{pertub1}
	Let $(M,\alpha,TM = I\oplus F)$ be a generalized $k$-contact structure, and let $B = (\beta_{ij})\in M_{k\times k}(\R)$. We define $\eta_i = \alpha_i - \sum_{j=1}^k\beta_{ij}\alpha_j$. Then, if $B$ is sufficiently small, $(M,\eta,TM = I\oplus F)$ is a generalized $k$-contact structure.
	
	In other words, we can understand the Reeb vector fields $X_1,\dots,X_k$ as a framing of the sub-bundle $I$, and we are allowed to take a different framing $Y_1,\dots,Y_k$ of $I$, as long as it is sufficiently close to the original and 
	$$Span_{\R}\{X_1,\dots,X_k\} = Span_{\R}\{Y_1,\dots,Y_k\}$$
\end{lemma}

\begin{proof}
	As $B$ is small, then $Id - B$ is invertible, then 
	$$\bigcap\ker \eta_ip
	= \bigcap\ker\alpha_i = F$$
	Also, we can write
	$$d\eta_i = (1- \beta_{ii})d\alpha_i + \sum_{j\neq i}\beta_{ij}d\alpha_j$$
	where $\beta_{ij}$ are small. And thus 
	$$d\eta_i^n = d\alpha_i^n + \omega$$
	where $\omega$ is small. As non degeneracy is an open condition, if $\omega$ is small enough, then $d\eta_i$ is non degenerate. Similarly, $d\eta_i^n = d\alpha_i^n + \omega_0$ where $\omega_0$ is small, and therefore $$\eta_1\wedge\dots\wedge\eta_k\wedge d\eta_j^n = \det(Id -B)\alpha_1\wedge\dots\wedge\alpha_k(d\alpha_i^n + \omega_0)$$
	which is non zero for small $\omega_0$.
	
	Finally, $I\subset \ker (d\eta_j)$, but as $TM =I\oplus F$ and $d\eta_j$ is non degenerate on $F$ (for small perturbation of the identity $B$), then $\ker (d\eta_j)\subset I$
\end{proof}

A similar argument proves the following Lemma
\begin{lemma}\label{anterior}
	Consider a manifold $M$ of dimension $2n+k$, a splitting, $TM = I\oplus F$, $dim I = k$ and, linear independent, non vanishing $1$-forms $\alpha_1,\dots,\alpha_k$ such that
	\begin{itemize}
		\item $F = \bigcap_{i}\ker(\alpha_i)$
		\item $d\alpha_k$ is non degenerate on $F$ and $I = \ker(d\alpha_k)$
		\item $I\subset\ker d\alpha_j$ for $j=1,\dots,k-1$
	\end{itemize} 
	Then, there exists $B = (\beta_{ij})\in GL(\R^{n-1})$ and a change of coordinates $\eta_{j} = \sum_{i}\beta_{ij}\alpha_{j}$ , $j=1,\dots,k-1$ such that $(M, \eta_1,\dots,\eta_{k-1},\alpha_k,TM = I\oplus F)$ is a generalized $k$-contact manifold
\end{lemma}

With some additional conditions, we can improve Lemma \ref{pertub1}
\begin{lemma}\label{repa}
	Under the conditions of the previous Lemma \ref{pertub1}, suppose that the action is topologically transitive.
	For $B\in Gl(\R^k)$, denote by $\eta = B\alpha$ the change of coordinates $\eta_i = \sum_j\beta_{ij}\alpha_j$. Then, there exists a Zariski  open subset of $\Lambda \subset Gl(\R^k)$ such that $(M,B\alpha,TM = I\oplus F)$ is generalized $k$-contact.
\end{lemma}

\begin{proof}
	Like in the previous lemma, we have 
	$$F = \bigcap_{i=1}^k \alpha_i = \bigcap_{i=1}^k\eta_i$$
	and $$I\subset \ker d\eta_j\enskip\enskip\forall j = 1,\dots,k$$
	we must show that for a Zariski  open set $\Lambda$, the conditions 
	\begin{enumerate}\label{cocon}
		\item {}${d\eta_j}_{|_F}$ is non degenerate for every $1\leq j\leq k$;
		\item {}$\eta_1\wedge\cdots\wedge\eta_k\wedge d\eta_j^n$ is a volume form.
	\end{enumerate}
	are satisfied.
	
	Now, denote $$d\alpha^J = d\alpha_1^{J_1}\wedge\dots\wedge d\alpha_k^{J_n}$$
	for any multi-index $J = (J_1,\dots,J_k)\in \N^k$.
	Then, we can write 
	$$d\eta_j^n = \sum_{|J| = n}Q_Jd\alpha^J $$
	where $Q_J$ is some polynomial in the variables $\beta_{ij}$ and $|J| = J_1 + \dots + J_k$
	
	Notice that $d\alpha^J$ is clearly a top form over $F$, and thus $d\alpha^J = f_Jd\alpha_k^n$ for some function $f_J$. As the action is topologically transitive, and the forms $d\alpha^J$ and $d\alpha_k$ are invariant by this action\footnote{We recall that we defined the contact action in such way that it leaves the defining 1-forms $\alpha_1,\dots,\alpha_k$ invariant, that is, $\L_{X_i}\alpha_j=0$. In particular, this implies that, for any $1\leq j_1,\dots,j_l\leq k$, the forms $d\alpha = d\alpha_{j_1}\wedge \dots \wedge d\alpha_{j_l}$ are also invariant by the action.}, this function is constant over $M$. Thus, we write
	$$d\eta_j^n = P_j(B)d\alpha_k^n$$
	where $P_j(B)$ is a polynomial on the coefficients $\beta_{ij}$ of $B$.
	We also write
	$$\eta_1\wedge \dots\wedge\eta_k\wedge d\eta_j^k = P_j(B)\det(B)\alpha_1\wedge\dots\wedge\alpha_k\wedge d\alpha_k^n$$
	And thus, conditions (\ref{cocon}) are satisfied when $P_j(B)\neq 0$. This polynomial is non zero, for $P_j(Id) = 1$, and thus, the conditions (\ref{cocon}) are met for $B$ in a Zariski  open set.   
\end{proof}

	

\section{Anosov Dynamics}
We recall the definition of an Anosov action and some useful results:

\begin{definition}\label{anodef}
	Consider a compact smooth manifold $M$ and a smooth action $\phi:\R^k\times M \to M$. This action is said to be Anosov if there exists an element $a\in \R^k$, called an Anosov element, such that $\phi^a$ acts on $M$ normally hyperbolically, that is, there exists a, $d\phi^a$ invariant, continuous, splitting $TM = E^+\oplus T\phi\oplus E^-$, where $T\phi$ is the distribution tangent to the orbits, such that, there exists positive constants $C,\lambda$ for which
	\begin{align}\label{estimates}
		\|d\phi^{ta}(u^+)\|&\leq Ce^{-t\lambda}\|u^+\|\enskip\forall t>0\enskip\forall u^+\in E^+\\
		\|d\phi^{ta}(u^-)\|&\leq Ce^{t\lambda}\|u^i\|\enskip\forall t<0\enskip\forall u^-\in E^-
	\end{align}
\end{definition}

\begin{theorem}[Spectral decomposition for Anosov actions, for example \cite{Arbieto}]\label{Spectral} 

	Let $A:\R^k\to Diff(M)$ be an Anosov action. Then, the non-wandering\footnote{Let $G\to Hom(M)$ be an action of a topological group $G$ on a topological manifold $M$. We shall say that a point $x\in M$ is non-wandering, if, for every neighbourhood $U$ of $x$ and for every compact subset $S\subset G$, there exists $g\in S^c$ such that $gU\cap U\neq\emptyset$. The set of non-wandering points is clearly invariant by the action of $G$.} set $NW(A)$ of $A$ is  a finite union $NW(A) = \Omega_1\cup\dots\cup\Omega_l$ of disjoint compact and invariant sets $\Omega_i$. Moreover, each $\Omega_i$ cannot be further subdivided in two compact, non-empty, disjoint and invariant subsets and the action $A$ on each $\Omega_i$ is topologically transitive on $\Omega_i$.
\end{theorem}

\begin{corollary}\label{tra}
	An Anosov action $\phi:\R^k\times M\to M$ that preserves a volume is topologically transitive.
\end{corollary}

\begin{proof}
	Let $x\in M$ and let $K\subset \R^k$ be a compact set, suppose that $K\subset B_R(0)\subset\R^k$, and let $g\in B_R(0)^c$. Then, $\phi(g,\cdot)$ defines a volume preserving diffeomorphism of $M$. Fix an arbitrary neighbourhood $U$ of $x$. From the Poincar\'e recurrence theorem, for almost every $y\in U$, there exists $n\in\N$ such that $\phi(n\cdot g, y)\in U$. In particular, $(n\cdot g)\cdot U\cap U\neq\emptyset$. As $(n\cdot g)\in B_{nR}(0)^c\subset B_R(0)^c\subset K^c$ we conclude that $x$ is non-wandering. Thus, $NW(\phi) = M$. From the connectedness of $M$, it follows from Theorem \ref{Spectral}, that $\phi$ is topologically transitive.
\end{proof}

\begin{lemma}[Structure of Anosov elements, for example \cite{Ba-Maq1}]\label{structurelemma}
    The set $\mathcal A(\phi)$ of Anosov elements forms an open set of $\R^k$, and each connected component of $\mathcal A(\phi)$, is an open cone. Moreover, Anosov elements on the same open cone will have the same invariant distributions.
\end{lemma}

\begin{definition}\label{genkcontAnosov}
	A generalized $k$-contact action $\phi$ will be called a contact Anosov action, if some Reeb vector field $X_j$ defines an Anosov element of the induced contact action. 
\end{definition}
\begin{lemma}\label{ASOA}
	If $\phi:\R^k\to M$ is a contact Anosov action, then there exists $B \in GL(\R^k)$ such that every Reeb vector field is Anosov and they have the same invariant splitting $TM = T\phi\oplus E^+\oplus E^-$
\end{lemma}
\begin{proof}
	As the action preserves the volume form $\alpha_1\wedge\dots\wedge\alpha_k\wedge d\alpha_j^n$, the action is topologically transitive. Thus, the previous lemma (Lemma \ref{repa}) implies that for almost any linear reparameterization of the action, the corresponding structure is still generalized $k$-contact. From Lemma \ref{structurelemma}, the set of Anosov elements is open, and moreover, the invariant splitting $TM = T\phi\oplus E^+\oplus E^-$ depends only on the choice of open cone, thus just choose a reparameterization that puts every Reeb vector field in the same open cone.
\end{proof}
\begin{definition}
	Let $(M,\alpha,TM = I\oplus F,\phi)$ be a contact Anosov action. The parameterization $\alpha$ will be called adapted if every Reeb vector field is Anosov.
\end{definition}
\begin{remark}
	The Lemma \ref{ASOA} proves that every generalized $k$-contact Anosov action admits an adapted parameterization. 
\end{remark}

\begin{remark}
	Another way to see this definition is to start with an Anosov $\R^k$-action $\phi$, consider the action as given by a family of commuting vector fields $X_j$, where each $X_j$ is the vector field associated with an Anosov element. Using the splitting $TM = T\phi\oplus E^+\oplus E^-$ we define the dual $1$-forms $\alpha_1,\dots,\alpha_k$, and we suppose that each of those forms have constant rank $2n+1$ and satisfies: 
	$$\alpha_1\wedge\dots\wedge \alpha_k\wedge d\alpha_j^n\enskip\text{is a volume form}.$$ 
\end{remark}

\begin{remark}
	Notice that $dim E^+ = dim E^- = n$. This follows from the fact that, for some fixed $j$, $d\alpha_j$ restricted to $E^+\oplus E^-$ is a $\phi$- invariant symplectic form, the hyperbolic dynamics will ensure that $E^\pm$ are Lagrangian subspaces.
\end{remark}

\begin{remark}
A contact Anosov action can not be a suspension of a $\Z^k$ Anosov action by diffeomorphisms. In fact, by Lemma \ref{transverse} there can be no global section transverse $\Lambda$ to the contact action.
\end{remark}

\section{The algebraic picture and main theorem}

First we recall some definitions from the classical theory of semisimple Lie algebras.

\begin{definition}
	Consider a real semisimple Lie algebra $\lieg$. A Cartan subspace $\liea$ of $\lieg$ is an abelian subalgebra, such that, for every $x\in \liea$, the linear map $ad(x)$ is hyperbolic (that is, it is $\R$-diagonalizable), and maximal for these properties. The rank of $\lieg$ is the dimension of $\liea$ and does not depends on the choice of Cartan Subspace.
\end{definition}
It is well known:
\begin{lemma}[\cite{helgason}]\label{cartansubpace}
Let $G$ be a semisimple Lie group with Lie algebra $\lieg$.If $\liea$ is a Cartan subspace of $\lieg$, then it's centralizer $Z_\lieg(\liea)$ can be written as
	$$Z_\lieg(\liea) = \liek\oplus\liea$$
	where $\liek$ is the Lie algebra of a compact subgroup $K\subset G$.
\end{lemma}
\begin{remark}
    The notation $\liek$ for the compact part of the centralizer is not standard in the literature, where it is more common to denote it by $\mathfrak m$. Here, however, we follow the notations of T. Barbot and C. Maquera (\cite{Ba-Maq3}).
\end{remark}

The following definition was first given by Hans-Christoph Hof \cite{HOF}, a more modern approach was later given by A. Katok and R. Spatzier \cite{kat}.
\begin{definition}
The Weyl chamber action associated with a semisimple Lie group $G$ is the right action of a Cartan subspace on the quotient $G\slash K$. This action was called Weyl chamber flow by A. Katok and R. Spatzier.
\end{definition}

The following theorem  shows that if $\Gamma$ is a uniform lattice acting freely on $G\slash K$, then the Weyl chamber action descends to an Anosov action on the compact manifold $\Gamma\backslash G\slash K$.

\begin{theorem}[\cite{Ba-Maq3}]\label{theo44}
Let $G$ is a connected Lie group with Lie algebra $\lieg$,  $K$ is a compact subgroup of $G$ with Lie algebra $\liek$,  $\liea$ is an abelian subalgebra of $\lieg$, contained in the normalizer $N_\lieg\liek$ of $\liek$ and such that $\liek\cap\liea = \{0\}$ and let $\Gamma$ is a uniform lattice in $G$ acting freely on $G\slash K$. Under this conditions we have an right action of $\liea$ on $\Gamma\backslash G\slash K$. This action is Anosov if, and only if, there exists $x\in \liea$ and a $ad(x)-$invariant splitting 
$$\lieg = \liek\oplus\liea\oplus \mathcal S\oplus\mathcal U$$
such that the eigenvalues of the $ad(x)$  action on $\mathcal S$ (resp. $\mathcal U$) has negative (resp. positive) real part.
\end{theorem}
 
 It is clear that for a Weyl chamber action, if we fix a Weyl chamber (notion of positive roots) the splitting into positive and negative rootspaces give us the desired  $\mathcal S,\mathcal U$ subspaces.\\

\subsection{An example: Definitions}
 We shall consider $SO(k,k+n)$ as the matrix subgroup of  $GL(\R^{2k+n})$ of elements that preserves the bilinear form $$\langle u,v\rangle = \sum_{i=1}^ku_iv_i - \sum_{i=k+1}^{2k+n}u_iv_i.$$
 
Using this identification, we can consider the Lie algebra $\mathfrak{so}(k,k+n)$ also as a matrix Lie algebra given by 

\begin{align}\label{sokknpresentation}
\mathfrak{so}(k,k+n) = \left\{
\left(
\begin{array}{ccc}
\mathfrak{so}(k)
&
C
&
X^t
\\
C^t
&
\mathfrak{so}(k)
&Z^t\\
X
&
Z
&
D
\end{array}
\right)
\;;\; 
C\in M_{k\times k};\;
X,Z	\in M_{n\times k};\;
D\in \mathfrak{so}(n)\right\}
\end{align}

Some computations shows that the set $\Ii \subset M_{k\times k}$ of diagonal matrices and embed on $\mathfrak{so}(k,k+n)$ in a natural way, that is
\begin{align}\label{cartansubspace}
\Ii = \left\{ \left(\begin{array}{ccc}
0&J&0\\
J&0&0\\
0&0&0
\end{array}\right)\in \mathfrak{so}(k,k+n)\;;\;J\in M_{k\times k}\;\;\text {is diagonal}
\right\}
\end{align}
is in fact a Cartan subspace for the Lie algebra $\mathfrak{so}(k,k+n)$. It is clear that $\Ii$ induces an action of $\R^k$ on $SO(k,k+n)$. The centralizer of $\Ii$ is in fact $\Ii\oplus\mathfrak{so(n)}$. Thus, this action descends to an action on $SO(k,k+n)\slash SO(n)$. Moreover, for any choice of Weyl chamber the corresponding $\mathcal S,\mathcal U$ spaces satisfy 
$$\mathcal E :=\mathcal S\oplus\mathcal U = \left\{
\left(
\begin{array}{ccc}
\mathfrak{so}(k)
&
C
&
X^t
\\
C^t
&
\mathfrak{so}(k)
&Z^t\\
X
&
Z
&
0
\end{array}
\right)
\;;\; 
C\in M_{k\times k}, diag (C)=0;\;
X,Z	\in M_{n\times k};\;
\right\}$$

\subsection{An example: The generalized $k$-contact structure}\label{secexam}

We will define specific left invariant forms on $SO(k,k+n)$, which by construction descends to left invariant forms on $SO(k,k+n)\slash SO(n)$. 

We consider the splitting $\lieg = \mathfrak(so(k,k+n)) = \mathfrak{so(n)}\oplus\Ii\oplus \mathcal E$. This splitting, allow us to immerse (by zero extensions) $ \Ii^*\subset \lieg^*$, and therefore,with left invariant 1-forms on $G$. Because $\Ii$ is a Cartan subspace, this extension is $SO(n)$-invariant and thus it descends to a left invariant form on   $G\slash K = SO(k,k+n)\slash SO(n)$.

Let $e_1,\dots,e_k$ be the canonical basis of $\Ii$, and $\alpha \in \Ii^*\subset\lieg^*$ which we consider as a 1-form on $G$. For left invariant vector fields $A,B$, we have

$$d\alpha(A,B) = -\alpha([A,B])$$

Thus, if we consider the distribution $E$ on $G$ induced by $\mathcal E\subset\lieg$, $d\alpha$ is non degenerate over $E$ if, and only if, the bilinear form $$\mathcal E\otimes\mathcal E \ni a\otimes b\mapsto -\alpha([a,b])$$
is non degenerate. Some calculations show us that:

\begin{itemize}\label{itens}
    \item $[\Ii,\lieg]\cap \Ii = 0$
    \item $[\mathcal E,\mathcal E]\cap \Ii = \Ii$
    \item $d\alpha$ is non degenerate over $\mathcal E\otimes\mathcal E$ if, and only if, $\alpha(e_i)\neq\pm\alpha(e_j)\neq 0$ for all $1\leq i,j\leq k$
\end{itemize}

Because the last condition is open, there exists a basis $\alpha_1,\dots,\alpha_k$ of $\Ii^*$ which satisfies this condition.

The first condition show us that $\Ii\subset \ker d\alpha_j$ for $1\leq j\leq k$.

It is clear that $\alpha_1,\dots,\alpha_k$ actually defines a left invariant generalized $k$-contact structure on $G\slash K$.

\begin{remark}
    The action constructed in our example does not comes from a left invariant $k$-contact structure in the sense of \cite{Bolle}.  In fact, the second condition means that, for $1$-forms $\alpha,\eta\in\Ii^*$, we have $d\alpha = d\eta \Leftrightarrow \alpha = \eta$ and thus, it is not possible to obtain a basis $\alpha_1,\dots,\alpha_k$ of $\Ii^*$ satisfying the conditions of the Definition \ref{genkcont}
\end{remark}

\subsection{An example: Computations}
On this subsection we make explicit the computations we hinted at the previous subsection, to construct the generalized $k$-contact structure on $SO(k,k+n)\slash SO(n)$.

Now, our goal is to define some specific left invariant one forms on $SO(k,k+n)\slash SO(n)$ and make some computations. 


We take $\alpha \in\Ii^*$, where $\Ii$ is the Cartan subspace given in \ref{cartansubspace}, and we consider $\alpha$ as a linear form on $\mathfrak{so}(k,k+n)$ by extending it to zero according to the obvious basis. By definition, this linear form is zero on $\mathfrak{so}(n)$ and also $SO(n)$-invariant\footnote{Just notice that $\Ii$ is in the normalizer of $\mathfrak{so}(n)$ and $\Ii\cap \mathfrak{so}(n) = \{0\}$}, thus, it defines a left invariant $1$-form on $SO(k,k+n)$ which descends to a $1$-form  on $SO(k,k+n)\slash SO(n)$.\\

In what follows, we will make frequent use of the different injections $M_{k\times k},\mathfrak{so}(k),M_{n\times k}\hookrightarrow \mathfrak{so}(k,k+n)$, and thus, we fix the following notation:
We denote by $C_{ij}\in M_{k\times k}$ the matrix whose only non zero coordinate is the $i,j$ coordinate, whose value is $1$ and, $E_{ij}=C_{ij}-C_{ji}$ for any $1\leq i\leq j\leq k$, we will denote
\begin{align*}
F_{ij} = \left(\begin{array}{ccc}
E_{ij}&0&0\\
0&0&0\\
0&0&0
\end{array}\right)\enskip\enskip;\enskip\enskip
G_{ij} = \left(\begin{array}{ccc}
0&0&0\\
0&E_{ij}&0\\
0&0&0
\end{array}\right)
\enskip\enskip;\enskip\enskip
H_{ij} = \left(\begin{array}{ccc}
0&C_{ij}&0\\
C_{ji}&0&0\\
0&0&0
\end{array}\right)
\end{align*}

We will also make no distinction between $X,Z\in M_{n\times k}$ and their images:
$$\left(\begin{array}{ccc}
0&0&X^t\\
0&0&0\\
X&0&0
\end{array}\right)
\enskip\enskip\text{ and }\enskip\enskip
\left(\begin{array}{ccc}
0&0&0\\
0&0&Z^t\\
0&-Z&0
\end{array}\right)
$$

We shall denote by $X_{ij},Z_{ij}\in M_{n\times k}$ the matrix with $1$ in the $i,j$ coordinate and zero in all others. 

We shall make our calculations on $SO(k,k+n)$. For left invariant vector fields $A,B$, we have $d\alpha(A,B) = -\alpha([A,B])$. Thus, for the purpose of our calculations, we are interested only in the diagonal portion of the $M_{k\times k}$ on $\mathfrak{so}(k,k+n)$, that is, for two given matrices $M_1,M_2\in \mathfrak{so}(k,k+n)$, in the expression of

$$[M_1,M_2] = \left(\begin{array}{ccc}
R_1&C&X^t\\
C^t&R_2&Z^t\\
X&Z&D
\end{array}\right), R_1,R_2\in \mathfrak{so}(k), D\in \mathfrak{so}(n), X,Z\in M_{n\times k}, C\in M_{k\times k}$$
we are only interested in the matrix $C$. With this notation, we shall write 
$$\llbracket  M_1,M_2 \rrbracket = C \tilde= \left(\begin{array}{cc} 0&C\\ C^t&0 \end{array}\right)$$

Moreover, as we are actually interested only in the diagonal elements of $C$, which can be written as linear combination of elements of $\Ii$,  we define
$$|\!\!\llbracket  M_1,M_2\rrbracket\!\!| = Diag\llbracket M_1,M_2\rrbracket = \sum_j a_jH_{jj}$$

Now, some computations: If we write

$$M_1 = \left(\begin{array}{ccc}
R_1&A&X^t\\
A^t&R_2&Z^t\\
X&Z&D
\end{array}\right)\enskip\text{and}\enskip
M_2 = \left(\begin{array}{ccc}
\tilde R_1&\tilde A&\tilde X^t\\
\tilde A^t&\tilde R_2&\tilde Z^t\\
\tilde X&\tilde Z&\tilde D
\end{array}\right)
$$
then
$$\llbracket  M_1,M_2 \rrbracket = R_1\tilde A + A\tilde R_2 + X^t\tilde Z - \tilde R_1 A - \tilde A R_2 - \tilde X^t Z$$

It is clear that for $\{s,t\}\neq\{i,j\}$ we have
$$|\!\!\llbracket  F_{ij},H_{st}\rrbracket\!\!|=|\!\!\llbracket  G_{ij},H_{st}\rrbracket\!\!|=0$$
moreover,
$$|\!\!\llbracket  
F_{ij},F_{st}\rrbracket\!\!|=|\!\!\llbracket  
G_{ij},G_{st}\rrbracket\!\!| = |\!\!\llbracket  H_{ij},H_{st}\rrbracket\!\!|=0$$

and therefore 
\begin{align*}
d\alpha^{2k(k-1)}&(F_{12},\dots,F_{k-1,k},G_{12},\dots,G_{k-1,k},H_{12},H_{21},\dots, H_{k-1,k},H_{k,k-1}) = \\
&=\pm\Pi_{i<j}d\alpha^2(F_{ij},G_{ij},H_{ij},H_{ji})
\end{align*}

Now, from $|\!\!\llbracket  F_{ij},G_{ij}\rrbracket\!\!|=0$ and $|\!\!\llbracket  H_{ij},H_{ji}\rrbracket\!\!|=0$, we obtain
\begin{align}\label{1111}
d\alpha^2(F_{ij},&G_{ij},H_{ij},H_{ji})=\\
 =2&d\alpha(F_{ij},H_{ij})d\alpha(G_{ij,}H_{ji}) - 2d\alpha(F_{ij},H_{ji})d\alpha(G_{ij,}H_{ij})\\
=2&\alpha(|\!\!\llbracket  F_{ij},H_{ij}\rrbracket\!\!|)\alpha(|\!\!\llbracket  G_{ij,}H_{ji}\rrbracket\!\!|)
-2
\alpha(|\!\!\llbracket  F_{ij},H_{ji}\rrbracket\!\!|)\alpha(|\!\!\llbracket  G_{ij,}H_{ij}\rrbracket\!\!|)
\end{align}

Now, 
\begin{align*}
|\!\!\llbracket  F_{ij},H_{ij}\rrbracket\!\!| & = |\!\!\llbracket  G_{ij,}H_{ji}\rrbracket\!\!|= -H_{jj}.\\
|\!\!\llbracket  F_{ij},H_{ji}\rrbracket\!\!| &=|\!\!\llbracket  G_{ij,}H_{ij}\rrbracket\!\!| =  H_{ii}
\end{align*}

Thus,
\begin{align}\label{1122}
d\alpha^2(F_{ij},G_{ij},H_{ij},H_{ji}) = 
2
\alpha(H_{jj})^2 - 2\alpha(H_{ii})^2
\end{align}
More computations:

$$\llbracket  X_{ij},Z_{st}\rrbracket = \delta_{is}C_{jt}$$
and thus,
$$|\!\!\llbracket  X_{ij},Z_{st}\rrbracket\!\!| = 0\enskip \enskip\text{if}\enskip j\neq t\text{ and }i\neq s$$

Moreover
$$
|\!\!\llbracket  X_{ij},X_{st}\rrbracket\!\!| = 
|\!\!\llbracket  Z_{ij},Z_{st}\rrbracket\!\!| = 0$$
So, 
\begin{align}\label{2222}
d\alpha^{kn}&(X_{11},Z_{11},\dots,X_{nk},Z_{nk}) = 
=\Pi_{\substack{1\leq s\leq n \\ 1\leq t\leq k}}d\alpha(X_{st},Z_{st}) \\
&= (-1)^{kn}\Pi_{\substack{1\leq s\leq n \\ 1\leq t\leq k}}\alpha(|\!\!\llbracket  X_{st},Z_{st}\rrbracket\!\!|)
=(-1)^{kn}\Pi_{1\leq t\leq k}\alpha(H_{tt})^n
\end{align}

Finally, from 
\begin{align*}
    |\!\!\llbracket  X_{ij},F_{st}\rrbracket\!\!| = 
|\!\!\llbracket  X_{ij},G_{st}\rrbracket\!\!| = |\!\!\llbracket  X_{ij},H_{st}\rrbracket\!\!| = 0\\
|\!\!\llbracket  Z_{ij},F_{st}\rrbracket\!\!| = 
|\!\!\llbracket  Z_{ij},G_{st}\rrbracket\!\!| = |\!\!\llbracket  Z_{ij},H_{st}\rrbracket\!\!| = 0
\end{align*}
It follows that
\begin{align*}
    d\alpha^{kn+2k(k-1)}(X,Z,F,G,H) &=\pm d\alpha^{kn}(X,Z)\Pi_{i<j}d\alpha^2(F_{ij},G_{ij},H_{ij},H_{ji})\\
    &=\pm\Pi_{\substack{i<j\\ 1\leq t\leq k}}\alpha(H_{tt})^n\big(
    2
\alpha(H_{jj})^2 - 2\alpha(H_{ii})^2\big)
\end{align*}

We conclude, that $d\alpha^{kn+2k(k-1)}\neq 0$ if, and only if, 
we choose an appropriate  $\alpha\in \Ii^*$ that satisfies 
\begin{align}\label{asd}
0\neq\alpha(H_{jj})\neq\pm \alpha(H_{ii})\enskip\enskip\forall i\neq j,
\end{align}

\subsection{General Cartan actions}
On this subsection we repeat what was done in the previous one, now for a general Cartan action. Explicitly, we want to prove the following Theorem \ref{geralsemi} bellow.

\begin{theorem}\label{geralsemi}
	Consider a real, connected, semisimple, non compact, Lie group $G$ with Lie algebra $\lieg$. Let $\liea$ be the Cartan subspace and $K\subset G$ the compact group associated with the compact part of the center of $\liea$. Consider also a uniform lattice $\Gamma$  in $G$ acting freely on $G\slash K$. Then there exists a (left invariant) generalized $k$-contact structure on $G\slash K$ such that the induced $k$-contact action is Anosov and it coincides with the Cartan action.
\end{theorem}

 Because we lack a concrete representation of the algebras involved, we use a structure theorem by Kammeyer, where he give an explicit multiplication table for real semisimple Lie algebras in terms of the root system. This multiplication table allow us to make similar computations. In fact, Kammeyer's proof give us the necessary computations, though he does not says it explicitly.
 
 The idea of the proof is to imitate the construction in section \ref{secexam}. More precisely, we shall construct a splitting $\lieg_0 =\liea\oplus\mathcal K\oplus\mathcal N^+\oplus\mathcal N^- $ where $\mathcal N^\pm$ corresponds to the eigenspaces of positive and negative roots, and $\mathcal K$ is the compact part of the centralizer of $\liea$. In particular, this decomposition will be invariant by the action of $\liea$. The next step is to construct the generalized $k$-contact structure, that is, we must find linearly independent 1-forms $\eta_1,\dots,\eta_k$ such that
			 \begin{enumerate}
			     \item $\eta_j$ is $\mathcal K$-invariant.
			     \item $\cap_j\ker(\eta_j) = \mathcal K\oplus\mathcal N^+\oplus\mathcal N^- $
			     \item $\liea\oplus\mathcal K\subset \ker(d\eta_j)$
			     \item $(d\eta_j^N)|_{\mathcal N^+\oplus\mathcal N^- }\neq 0$, where $N = dim(\mathcal N^\pm)$.
		\end{enumerate}
			 
			 Using the idea of section \ref{secexam}, we shall take $1$-forms on $\liea^*$ and extend them to $1$-forms on $\lieg_0$. By construction, this will imply conditions (1), (2) and (3).  Finally, Kammeyer's theorem give us the computations necessary to show that condition (4) is actually an open condition, and we can therefore choose appropriate $1$-forms $\eta_1,\dots,\eta_k$ such that they define a basis of $\liea^*$ and, therefore, give us the desired generalized $k$-contact structure.\\

 We recall some basic definitions of the theory of semisimple Lie algebras to establish some notations before we state Kammeyer's result and indicate the step of the proof that give us our calculations.


\begin{lemma}[\cite{helgason},\cite{kamm}]\label{real}
		Consider $\lieg^0$ a real semisimple Lie algebra, a Cartan involution $\theta$ on $\lieg^0$, and the corresponding Cartan decomposition $\lieg^0 = \liek\oplus\liep$.
		
		There exists a maximal abelian, $\theta$-stable subalgebra $\lieh^0\subseteq\lieg^0$ such that $\liea = \lieh^0\cap\liep$ is a maximal abelian subalgebra in $\liep$.
\end{lemma}
\begin{lemma}[Helgason,\cite{helgason}, Chapter XI]\label{Lemmahelg}
	Under the above notations, $\lieh^0$ is a Cartan subalgebra of $\lieg^0$ and $\liea$ is a Cartan subspace
	of $\lieg^0$.
\end{lemma}

\begin{definition}
   With the notations above, we define, for every linear functional $\alpha\in \liea^*$:
    $$\lieg^0_\alpha := \{x\in\lieg^0\;;\; [h,x]=\alpha(h)x \;\forall h\in\liea\}$$
    If $\lieg^0_\alpha$ is non zero, we say that $\alpha$ is a restricted root of $(\lieg^0,\liea)$ and  $\lieg^0_\alpha$ its restricted root space.
    
    The set of restricted roots will be denoted by $\Phi(\lieg^0,\liea)$.\\
    
    In a similar way, if $\lieg$ is a complex semisimple Lie algebra and $\lieh \subset\lieg$ is a Cartan subalgebra, for every $\alpha\in\lieh^*$ we define
     $$\lieg_\alpha := \{x\in\lieg\;;\; [h,x]=\alpha(h)x \;\forall h\in\lieh\}$$.
    
    If $\lieg_\alpha$ is non zero, $\alpha$ is called a root of $(\lieg,\lieh)$ e $\lieg_\alpha$ the associated root space. The set of roots will be denoted by $\Phi(\lieg,\lieh)$
\end{definition}

\begin{definition}
	The set $\{Y\in\lieh\;;\;\alpha(Y)\neq 0\;\;\forall\alpha\in\Phi(\lieg,\lieh)\}$ has a finite number of connected components called Weyl chambers. For a fixed choice of Weyl chamber $\mathcal W$, we can define 
	$$ \Phi^+(\lieg,\lieh) = \Phi^+(\lieg,\lieh, \mathcal W) = \{\alpha\in  \Phi(\lieg,\lieh)\;;\;\alpha(x)>0\;\;\forall x\in \mathcal W\}$$
	
	The roots on $\Phi^+(\lieg,\lieh)$ are called positive roots, with respect to $\mathcal W$, and we speak about "choice of positivity", meaning that a Weyl chamber was chosen and the positive roots are those that are positive with respect to this Weyl chamber.\\
	
	In a similar way, we can make a choice of positivity for restricted roots and define the set $\Phi^+(\lieg^0,\liea)$ of positive restricted roots.
\end{definition}

\begin{definition}
	A positive root (resp. restricted root) is called simple if it is not the sum of two other positive roots (resp. restricted root). The set of simple positive roots will be denoted by $\Delta(\lieg,\lieh)$, (resp. $\Delta(\lieg^0,\liea)$).
\end{definition}

Consider a real semisimple Lie algebra $\lieg^0$ with a Cartan involution $\theta$, and corresponding Cartan decomposition $\lieg^0 =\liek\oplus\liep$. Consider the $\theta$-stable Cartan subalgebra $\lieh^0$ obtained in Lemma \ref{real}. 	Let $\lieg = \lieg^0_\C$ the complexification of $\lieg^0$, then,  $\lieh = \lieh^0_\C$ (the complexification of $\lieh^0$)  is a  Cartan subalgebra of $\lieg$. There exists a unique linear extension of the Cartan involution $\theta$ to $\lieg$ which we will also denote by $\theta$. It is clear that $\lieh$ is $\theta$-stable.\\

\begin{remark}[Kammeyer, H.,\cite{kamm}]\label{Restri}
The terminology "restricted root" and "restricted root space" is justified by the following: Consider the inclusion map $j:\liea\to\lieh$. Let $\Sigma = \{\alpha\in\Phi(\lieg,\lieh)\;;\;j^*\alpha\neq 0\}$ be the set of roots which doesn't vanishes on $\liea$, then:
\begin{itemize}
    \item $\Phi(\lieg^0,\liea) = j^*\Sigma$

    \item For each $\beta\in \Phi(\lieg^0,\liea)$, if we denote $\Sigma_\beta = \{\alpha\in\Sigma\;;\;j^*\alpha = \beta\}$, then
    $$\lieg^0_\beta = \big(\bigoplus_{\alpha\in\Sigma_\beta}\lieg_\alpha\big)\cap\lieg^0$$
\end{itemize}

From here on, we will always choose positive roots and restricted roots such that $j^*\Phi^+(\lieg,\lieh) = \Phi^+(\lieg^0,\liea)\cup\{0\}$.
\end{remark}

\begin{lemma}[Wissen, \cite{Wis} Corollary 2.38]
With the above notations, if we consider the $\theta$ decomposition $\lieh^0 = \liek^0\oplus\liea$ ($\liek^0\subset\liek$ and $\liea\subset\liep$). Then every root $\alpha\in \Phi(\lieg,\lieh)$ is real valued on $\liea\oplus i\liek^0$.
\end{lemma}
\begin{definition}
We say that a root $\alpha\in \Phi(\lieg,\lieh)$ is called
\begin{itemize}
    \item \textbf{Real:} If $\alpha$ is real valued on all $\lieh$, or equivalently, if $\alpha$ vanishes on $\liek^0$
    \item \textbf{Imaginary:} If $\alpha$ assumes purely imaginary values on all $\lieh$, equivalently, if $\alpha$ vanishes on $\liea$
    \item \textbf{Complex:} If $\alpha$ is neither Real nor Imaginary.
\end{itemize}
The sets of Real, Imaginary and Complex roots are denoted by $\Phi_\R$, $\Phi_{i\R}$ and $\Phi_\C$ respectively.
\end{definition}
\begin{remark}
Using the notations of Remark \ref{Restri}, it is easy to see that $\Sigma = \Phi_\C\cup\Phi_\R$.
\end{remark}

Denote by $\sigma$  the anti-linear complex automorphism of $\lieg$ given by the conjugation with respect to $\lieg^0$, that is,
		\begin{align*}
		\sigma:\lieg = \lieg^0\oplus i\lieg^0&\to \lieg^0\oplus i\lieg^0\\
		g = g_1 + ig_2&\mapsto \sigma(g)= g_1 - ig_2
		\end{align*}
It is clear that $\lieh$ is $\sigma$-stable, and thus, for $\lambda \in \lieh$, we can define $\lambda^\sigma(x) = \overline{\lambda(\sigma (x))}$. \begin{remark}[Kammeyer, \cite{kamm}, Subsection 3]
If $\alpha \in \lieh^*$ is a root then $\alpha^\sigma$ is also a root. We can thus construct a subset $\Phi^*(\lieg,\lieh)\subset \Phi(\lieg,\lieh)$  such that
$$\Phi^*(\lieg,\lieh)\cap \{\alpha,\alpha^\sigma\}\enskip\text{has cardinality 1 for every }\alpha\in \Phi(\lieg,\lieh)$$

For any set of roots $S\subset\Phi(\lieg,\lieh)$ , we denote $S^* = S\cap\Phi^*(\lieg,\lieh)$.
\end{remark}
	
\begin{theorem}[Kammeyer, \cite{kamm}, Theorem 4.1]\label{kamthe}
Let $\Delta(\lieg,\lieh)$ be the set of simple roots. There exists a partition $\Delta(\lieg,\lieh) =   \Delta^1\cup\Delta^0$ such that 
			$\mathcal H^1 = \{H_\alpha^1\;;\;\alpha\in \Delta^1\}$ and $\mathcal H^0 = \{H_\alpha^0\;;\;\alpha\in \Delta^0\}$ are basis of $\liea$ and $\lieh^0\cap\liek$
			 and a basis $\mathcal B$ of $\lieg^0$ given by
			\begin{align*}
			\mathcal B :=\{X_\alpha^0,X_\alpha^1\;;\; \alpha\in\Phi_{i\R}^+\cup\Phi_\C^*\}\cup\{Z_\alpha\;;\;\alpha\in\Phi_\R\}\cup\mathcal H^1\cup\mathcal H^2
			\end{align*}
			such that, 
			\begin{enumerate}
				\item $[H_\alpha^i,H_\beta^j] =0$ for every $\alpha,\beta$.
				\item $[H_\alpha^i,X_\beta^j] = c_{\alpha,\beta}^{ij}X_{\beta}^{i + j + 1}$.
				\item $[H_\alpha^i,Z_\beta  ] = d_{\alpha,\beta}^iZ_\beta$.
				\item $[Z_\alpha,Z_{-\alpha}] = -sgn(\alpha)\hat H^1_\alpha$  where $2\hat H_{\alpha}^1$ is a non zero $\Z$-linear combination in $\mathcal H^1$.
				\item\label{asss} $[X_\alpha^i,X_{-\alpha}^j] = (-1)^{ij}H_\alpha^{1 + i + j}$, for $\alpha\in\Phi_{\C}^*$, where $2H_{\alpha}^0$ and $H_\alpha^1$ are non zero $\Z$-linear combinations in $\mathcal H^0$ and $\mathcal H^1$ respectively.
				\item\label{assss} $ [X_\alpha^0,X_\alpha^1] = \tilde H_\alpha^0$, for $\alpha \in \Phi_{i\R}^+$, where $\tilde H_\alpha^0$ is a nonzero $\Z$ linear combination of elements $H_\beta^0$, where $\beta\in\Delta^0\cap\Delta_{i\R}$.
			
				\item For $\beta \not\in\{-\alpha,-\alpha^\sigma\}$; $[X_\alpha^i,X_\beta^j] = (-1)^{ij}\gamma_{\alpha,\beta}X^{i+j}_{\alpha + \beta} + sgn(\alpha)\gamma_{\alpha^\sigma,\beta}X^{i+j}_{\alpha^\sigma + \beta}$.
				\item $[Z_{\alpha},X_{\beta}^i] = [X_{\alpha}^{\frac{sgn(\alpha) -1}{2}},X_{\beta}^i]$.
				\item $[Z_\alpha,Z_{\beta}] = \hat Z_{\alpha,\beta}$ where $\hat Z_{\alpha,\beta}$ is some half integer combination of $Z_{\alpha + \beta}$, $Z_{\alpha^\sigma + \beta}$ and $Z_{\alpha + \beta^\sigma}$.
			\end{enumerate}
		
			and the constants $c_{\alpha,\beta}^{ij}$, $d_{\alpha,\beta}^i$ and $\gamma_{\alpha,\beta}$ are non zero half integers and satisfy.
			\begin{enumerate}
			    \item[(i)] $d^0_{\alpha,\beta}=0$ and $d_{\alpha,\beta}^1=d_{\alpha,-\beta}^1\neq 0$
			    \item[(ii)] $c_{\alpha,\beta}^{1j} \neq 0$ if $\beta\not\in \Phi_{i\R}$.
			    \item[(iii)] $c_{\alpha,\beta}^{1j} = 0$ if $\beta\in \Phi_{i\R}$.
			    
			    \item[(iv)] $c_{\alpha,\beta}^{0j} = -c_{\alpha,\beta}^{0,j+1}\neq 0$
			
				\item[(v)] $\gamma_{\alpha,\beta} = \gamma_{-\alpha,-\beta}$
				\item[(vi)] $\gamma_{\alpha,\beta} = \pm (r+1)$ where $r$ is the largest integer such that $\beta -r\alpha\in\Phi(\lieh,\lieg)$
				
			\end{enumerate}
			
			Moreover, there exists an involution $\tau$ of $\lieg$ such that, if we denote  
			\begin{align*}
		\mathcal N^+ &:= 
		        Span_\R\{X_\alpha^i,Z_\beta, \alpha\in \Phi_{\C}^{*+},\beta \in\Phi_{\R}^+\}\\
		\mathcal K&:= 
			    Span_\R\{\mathcal H^0\}\oplus
			    Span_\R\{X_\alpha^0,X_\alpha^1\;;\;\alpha\in\Phi_{i\R}^+\}\\
		\mathcal K'&:=
		    	Span_\R\{X^i_\alpha +(-1)^i X^i_{-\alpha},Z_\gamma +  Z_{-\gamma}, \alpha\in \Phi_\C^{*+},\gamma \in \Phi_\R^+\}\\
		\tilde{\mathcal K}&:=
		        \mathcal K\oplus \mathcal K'
			\end{align*}
			then
			$$\lieg^0 = \tilde {\mathcal K}\oplus\liea\oplus\mathcal N^+$$
			is a Iawasawa decomposition of $\lieg^0$.
		
		\end{theorem}
			\begin{remark}
				On items \ref{asss} and \ref{assss} on Theorem \ref{kamthe} above, we have stated that $\mathcal H^0$, $\mathcal H^1$, $\mathcal H^0_\alpha$ and $\mathcal H^1_\alpha$ are non zero. These affirmations were not stated explicitly on Kammeyer's paper (\cite{kamm}). They are, however, a clear consequence of the proof. Those vectors are obtained by constructive methods, and a careful analysis of the construction show them to be non zero.
			\end{remark}
			
			We are now ready to prove Theorem \ref{geralsemi}

			\begin{proof}[Proof of Theorem \ref{geralsemi}]
		
		 As we remarked before, Theorem \ref{theo44} implies that our action is Anosov, moreover, almost\footnote{Here, almost every element means that the complement is a finite union of hyperplanes.} every element of $\liea$ is Anosov.  Thus it remains to choose the $1$-forms.\\

			 First let us notice that Kammeyer's theorem give us, the restricted root system $\Phi(\lieg^0,\liea)$. That is, for every $\beta\in \Phi_{i\R}^+\cup\Phi_\C^*$, $\gamma \in \Phi_\R$, $j\in\{0,1\}$,  we define linear functionals $\lambda_\beta^j,\lambda_\gamma:\liea\to \R$ by:
			 \begin{align*}
			   \lambda^j_\beta(H_\alpha^1) &= c_{\alpha,\beta}^{1j}\\
			   \lambda_\gamma(H_\alpha^1) &= d_{\alpha,\beta}^{1}\enskip\enskip\enskip\forall \alpha\in\Delta^1.
			 \end{align*}
		
It is clear that $\lambda_\beta^j,\lambda_\gamma$ are restricted roots, with corresponding restricted rootspaces generated by $X_\beta^j$ and $Z_\gamma$. 
Moreover, as $\lieg^0 = \tilde{\mathcal K}\oplus\liea\oplus\mathcal N^+$
			is a Iawasawa decomposition of $\lieg^0$, it follows, that, we can choose a notion of positivity such that 
			$$\Phi^+(\lieg^0,\liea) = \{\lambda^j_\beta,\lambda_\gamma\;;\;\beta\in\Phi_\C^{*+},\gamma\in\Phi_\R^+\}$$
			
			Let us denote $$\Phi^\pm(\lieg^0,\liea) = \{\lambda_\beta^i,\lambda_\gamma, \beta\in \Phi_{\C}^{*\pm},\gamma \in\Phi_{\R}^\pm\}$$
		
			We have the restricted root space decomposition 
			\begin{align}\label{decom}
			\lieg^0 = \overbracket{\lieg_0^0}^{\liea\oplus\mathcal K}\oplus \overbracket{\bigoplus_{\lambda\in\Phi^+ }\lieg_\lambda^0}^{\mathcal N^+}\oplus\overbracket{ \bigoplus_{\lambda\in\Phi^- }\lieg_\lambda^0}^{\mathcal N^-}
			\end{align}
			As we remarked on Lemma \ref{Lemmahelg}, $\liea$ is a Cartan subspace, moreover, from Kammeyer's theorem, $\mathcal K$ is contained in the centralizer of $\liea$. As $\lieg_0^0 = \liea\oplus\mathcal K$ it follows that $\mathcal k$ is in fact the compact part of the centralizer of $\liea$.\\
			
			We are now, ready to define the  linear functionals.\\

			Take $\lambda\in \liea^*$. We can use the decomposition (\ref{decom}) to extend $\lambda$ to a linear functional on $\lieg^0$, and therefore we can understand it as a left invariant $1$-form on $G$ (the Lie group with Lie algebra $\lieg^0$ we started with). Moreover, as the Lie algebra $\mathcal K$ is in the centralizer of $\liea$, than $\lambda$ is right invariant by $K\subset G$ (the Lie subgroup corresponding to $\mathcal K$) and, therefore, can be seen as a left invariant $1$-form on $G\slash K$. Because it is left invariant, it pass on to the quotient $\Gamma\backslash G\slash K$. \\
			
			Notice that for such $\lambda$ the items (1), (2) and (3) of Kammeyer's theorem ensures that, for any $Z\in \mathcal B$ we have $[\mathcal H^0\oplus\mathcal H^1,Z]\cap\liea = \{0\}$. From (6), (8) and (7), it follows, that for $\beta\in \Phi_{i\R}^+$ we also have $[X_\beta^j,Z]\cap\liea = \{0\}$
			and thus, for any $Y\in \lieg_0^0$
			$$d\lambda(Y,Z) = -\lambda([Y,Z])=0$$
			that is $\lieg_0^0\subset\ker d\lambda$.\\

			Now we  want conditions to show that the inequality above is actually an equality, that is $\ker (d\lambda) = \lieg_0^0 = \liea\oplus \mathcal K$ and thus is non degenerate on $\mathcal N^+\oplus\mathcal N^-$. We just have to choose $\lambda$ such that
			\begin{itemize}
			    \item [(A)] $\lambda$ does not vanishes on $H^1_\alpha$, $\hat H^1_\alpha$\footnote{Such $\lambda$ does exists, for $H^1_\alpha$, $\hat H^1_\alpha$ are always non zero and are finite in number.}.
			\end{itemize} 
			
		Let us denote $X^i= (X^i_\alpha)_{\alpha\in \Phi_\C^*}$ and $Z = (Z_\beta)_{\beta\in\Phi_\R}$. Much like in the example, we want to compute $d\lambda^N(X^0,X^1,Z)$ where $N$ is de dimension of $\mathcal N^+$. Following the computations of the example, we denote $|\!\!\llbracket A,B\rrbracket\!\!|$ the projection of $[A,B]$ into the Cartan subspace $\liea$. It follows from Kammeyer's theorem that
\begin{align*}
    |\!\!\llbracket  Z_\alpha,X_{\beta}^i\rrbracket\!\!|&=0\\
    |\!\!\llbracket  Z_\alpha,Z_{\beta}\rrbracket\!\!|&=0\enskip\enskip \alpha\neq -\beta\\
    |\!\!\llbracket  X_\alpha^i,X_{\beta}^j\rrbracket\!\!|&=0 \enskip\enskip \alpha\neq -\beta\\
    |\!\!\llbracket  X_\alpha^i,X_{\alpha}^j\rrbracket\!\!|&=0 \enskip\enskip i\neq j
\end{align*}
From this computations, it follows that
$$
    d\lambda^N(X^0,X^1,Z)=\pm\Pi_{\substack{\beta \in \Phi_\C^{*}\\ \gamma \in \Phi_\R}}
    d\lambda(X_\beta^0,X_{-\beta}^0)d\lambda(X_\beta^1,X_{-\beta}^1)d\lambda(Z_\gamma,Z_{-\gamma})
$$

But from condition $(A)$ we have
\begin{align*}
     d\lambda(X_\beta^0,X_{-\beta}^0)&=\lambda(H_\beta^1)\neq 0\\
     d\lambda(X_\beta^1,X_{-\beta}^1)&=-\lambda(H_\beta^1)\neq 0\\
     d\lambda(Z_\gamma,Z_{-\gamma})&=-sgn(\gamma)\lambda(\hat H_\gamma^1)\neq 0
\end{align*}
and thus $d\lambda^N(X^0,X^1,Z)\neq 0$ as we desired.

			Notice that the condition (A) is open, therefore, we can choose a basis $\eta_1,\dots,\eta_k$ of $\liea^*$ satisfying it. From previous considerations, it is clear that such 1-forms descends to left invariant $1$-forms on $G\slash K$ which then defines a generalized $k$-contact structure on $\Gamma\backslash G\slash K$ for any uniform lattice $\Gamma$.
		\end{proof}

	\section{General Algebraic actions}
	
	On this section we will use Barbot-Maquera's \cite{Ba-Maq3} classification of algebraic Anosov actions to extend our previous result to a more general class of algebraic actions.
	
	First recall a couple of definitions that tell us how to construct new algebraic actions:
	
	\begin{definition}\label{def1}
	    Let $(G,K,\Gamma,\mathfrak a)$ be an Weyl chamber action, and suppose that $K$ is not semisimple, then, it's Lie algebra $\liek$ splits as
	    $$\liek = \mathcal {T_e}\oplus \liek'$$
	    where $\mathcal {T_e}$ is an abelian ideal (in fact the nilradical), and $\liek^*$ a compact semisimple Lie algebra (in fact the Levi part of $\liea\oplus \liek$). 
	    
	    Now we consider $\liek'\subset\liek$  any subalgebra which contains the Levi factor $\lieg^*$ and $K'\subset K$ the associated subgroup. Let $\liea^*\subset \liek$ any supplementary Lie algebra for $\liek'$ and denotes by $\liea' = \liea\oplus\liea^*$. 
	    
	    A modified Weyl chamber action is the algebraic action given by $(G,K',\Gamma,\liea')$.
	\end{definition}
	
	\begin{definition}\label{def2}
	    Let  $(G',K',\Gamma',\mathfrak a')$ be an algebraic actions. It is said to be a central extension of $(G,K,\Gamma,\mathfrak a)$ if 
	    \begin{itemize}
	        \item{}  We have a central exact sequence $$1\to H_0 \to G'\stackrel{p}{\to} G\to 1$$
	        where the Lie algebra $\lieh_0$ of $H_0$ is contained in $\liea'$.
	        \item{} $p(K') = K$
	        \item{} $p(\Gamma') = \Gamma$
	        \item{} $p_*\liea' = \liea$
	    \end{itemize}
	\end{definition}
	
	We can now state a simpler version of  C. Maquera and T. Barbot's classification theorem:
	\begin{theorem}
	    Let  $(G,K,\Gamma,\mathfrak a)$ be an algebraic Anosov action. Then,
	    \begin{itemize}
	        \item either $(G,K,\Gamma,\mathfrak a)$ is commensurable\footnote{Commensurability is an equivalence relation between algebraic actions introduced by C. Maquera and T. Barbot that implies that implies that, up to finite coverings, they are conjugated one to another. For details, see \cite{Ba-Maq3}} to a central extension over a (modifyed) Weyl chamber action (this happens if $G$ is reductive)
	        \item either $(G,K,\Gamma,\mathfrak a)$ is commensurable to (nil)-suspension.
	    \end{itemize}
	\end{theorem}
	
	Thus, to prove Theorem \ref{TheoremB} we must show that the constructions \ref{def1} and \ref{def2} also have compatible generalized $k$-contact structures.

	    \begin{lemma}\label{modweyl}
	        Let $(G,K',\Gamma,\liea')$ be a modified Weyl chamber action. Then $\Gamma\backslash G\slash K'$ admits a generalized $k+l$ contact structure whose induced contact action coincides with the $\liea'$ action.
	    \end{lemma}
	    \begin{proof}
	        It is clear that the bundle $\pi:\Gamma\backslash G\slash K'\to \Gamma\backslash G\slash K$ is a principal torus bundle.  In general, however, it doesn't need to admits a flat connection. But, the additional (homogeneous) structure on the base space allow us to modify our previous construction (Proposition \ref{PBundle}) to our benefit. We denote by $\alpha_1,\dots,\alpha_k$ the (left invariant) 1-forms on $G\slash K$ that give us a generalized $k$-contact structure on $G\slash K$, which descends to a generalized $k$-contact structure on $\Gamma\backslash G\slash K$. We also denote by $X_1,\dots,X_k$ de Reeb vector fields. For a fixed connection, we denote by $\hat X_1,\dots,\hat X_k$ the lifted vector fields.
	        
	        Let's recall some details of the construction in Proposition \ref{PBundle}.
	        
	        We choose a basis of $\R^l$ and wrote our connection as a $\R^l = \liea^*$ valued 1-form $(\xi_1,\dots,\xi_l)$. The canonical vertical  vector fields $Y_1,\dots,Y_l$ satisfy $\xi_i(Y_j) = \delta_{ij}$. Afterwards, we showed that $Y_j \in ker (d\xi_i)$ for every $i,j$.
	        
	        The flatness of the connection was used to show that in fact every horizontal field is in the kernel of $d\xi_i$. But we need a weaker conclusion. In fact, we need only to show that the lifted vector fields $\hat X_j$ are in the kernel of $d\xi_i$. This will allow us to satisfy the hypothesis of Lemma \ref{anterior}.\\
	       
	        First let's fix some notations. As we have seem, we can decompose  $Lie(G) = \lieg$ as
	        $$\lieg = \liek'\oplus\liea^*\oplus\liea\oplus\mathcal S\oplus\mathcal U$$
	        
	        We denote by $E$ the distribution over $G\slash K$ obtained by left translating $\mathcal S\oplus\mathcal U$, and by $\hat E$  the distribution over $G\slash K'$ obtained in the same way. We notice that for any $j$, we have that $d\pi^*\alpha_j = \pi^*(d\alpha_j)$ is non degenerate on $\hat E$, for $d\pi$ induces an isomorphism between $\hat E$ and $E$.
	        
	        We also denote by $\mathcal I$ the distribution generated by $Y_1,\dots,Y_l,\hat X_1,\dots,\hat X_k$, that is, $I$ is the left translation of $\liea^*\oplus\liea$.\\

	        Now we construct a connection. We decompose $Lie(G) = \lieg$ as
	        $$\lieg = \liek'\oplus\liea^*\oplus\liea\oplus\mathcal S\oplus\mathcal U$$
	        notice that this splitting is $\liek'$ invariant (it is $\liea$ invariant and $\liea$ commutes with $\liek'$).
	        We consider the projection (using the above splitting)
	        $$\lieg\slash \liek' \to \liea^*$$
	        
	        Because the splitting is $\liek'$ invariant, the above projection defines a $\liea^*$ valued $1$-form on $G\slash K'$, and thus a connection of the principal bundle $G\slash K'\to  G\slash K$.
	        
	        By construction, the distribution $\hat E$ is horizontal, and thus $(\xi_i)_{|_{\hat E}} = 0$. 
	        
	        Because this connection is left-invariant, $d\xi_i$ can be completely understood by it's behavior on a single point. Let $X,Y\in \lieg\slash \liek'$, then
	        $$d\xi_i(X,Y) = -\xi_i([X,Y])$$
	        
	        The Anosov property of our action means that for any left invariant vector field $Z \in \hat E$ $[\hat X_j,Z]\subset \hat E$. This means that $d\xi_i(X_j,\;\cdot\;) = 0$ and thus $I \subset ker(d\xi_i)$.\\
	        
	        We are now ready to apply Lemma \ref{anterior}: We consider on $M = G\slash K'$ the 1-forms given by:
	        $$\eta_1 = \pi^*\alpha_1\;,\;\dots\;,\;\eta_k = \pi^*\alpha_k,$$ 
	        $$\eta_{k+1} = \xi_1 + \pi^*\alpha_1 \;,\;\dots\;,\;\eta_{k+l} = \xi_l + \pi^*\alpha_1$$
	        and the splitting
	        $$TM = I\oplus \hat E$$
	        
	        It is clear that $E = \cap_{j=1}^{k+l} ker(\eta_j)$ and we have shown that for $j\leq k$ we have $ker (d\eta_j) = I$. We also have shown that for $j\geq k+1$ we have $I\subset \ker (d\eta_j)$.
	       
	       As every 1-form considered is left invariant, this construction (and the lemma's conclusion) descends to our desired generalized $k+l$-contact structure on $\Gamma\backslash G\slash K'$. That the contact action coincides with the $\liea'$ action should be clear.

	    \end{proof}

	\begin{lemma}
	    Let  $(G,K,\Gamma,\mathfrak a)$ be a (Modified) Weyl chamber action and $(G',K',\Gamma',\mathfrak a')$ be a central extension. Then $\Gamma'\backslash G'\slash K'\to \Gamma\backslash G\slash K$ is a principal torus bundle that admits a flat connection.
	\end{lemma}
	
	\begin{proof}
	    First we notice that because $\lieg$ is semisimple, the induced exact sequence
	    $$0\to \lieh_0\to \lieg'\stackrel{p}{\to} \lieg \to 0$$
	    splits. Moreover, as the extension is central, we can identify \begin{align}\label{ident}
	        \lieg' = \lieg\oplus\lieh_0.
	  \end{align} 
	  
	  Now, we do as we've done in Lemma \ref{modweyl}, that is, we induce a left invariant connection on the bundle $G'\slash K'\to G\slash K$ by projecting left invariant vector fields $\lieg'\slash \liek'\to \lieh_0\slash \liek'$. The identification \ref{ident} allow us to compute the curvature of this connection and to check that it is flat.
	\end{proof}
	
	\begin{corollary}[Theorem \ref{TheoremB}]
Let $(G,K,\Gamma,\mathfrak a)$ be an algebraic Anosov action which is not (commensurable to) a suspension. Then there exists an generalized $k$-contact structure on $\Gamma\backslash G\slash K$ such that the induced contact action is Anosov and it coincides with the algebraic action.
	\end{corollary}

\section{Final Remarks}

We have proved, using the classification of algebraic Anosov actions given by T. Barbot and C. Maquera \cite{Ba-Maq3}, that every algebraic Anosov action of $\R^k$, which is not a suspension, does admits a compatible generalized $k$-contact structure. This shows that it is not unreasonable to suppose that algebraic Anosov actions can be taken as local models of contact Anosov actions. This, along with the work of Y. Benoist, P. Foulon and F. Labourie's \cite{BFL},  motivates the work at \cite{Almeida}, where it is shown that a contact Anosov action, $\phi:\R^k\times M\to M$, with smooth invariant bundles is (up to conjugation) affine, that is, there exists a Lie group $G$, a closed subgroup $H$ and a discrete subgroup $\Gamma$ such that $M = \Gamma\backslash G\slash H$ and the action is given by the right multiplication of a subgroup $A$ in the normalizer of $H$\footnote{Notice that an algebraic action is an affine action where the subgroup $H$ is compact, and the subgroup $\Gamma$ is an uniform lattice.}.

Keeping in mind the Kalinin-Spatzier \cite{kal} conjecture that states that Anosov action of $\R^k$, for $k\geq2$, are algebraic, we are lead to the following question:
\begin{itemize}
    \item[] Consider an Anosov action  of $\R^k$, for $k\geq2$ with smooth invariant bundles. If it is not a suspension, does it does admits a compatible generalized $k$-contact structure?
\end{itemize}

To give a positive answer for this question is to give a nice step towards showing Kalinin-Spatzier conjecture. However, to answer negatively, is to show that, in the smooth invariant bundles context, the Kalinin-Spatzier conjecture is false.

\bibliographystyle{amsplain}

\begin{thebibliography}{10}

    \bibitem{Almeida}
Matos de Almeida,U. N.
Algebricity of Contact Anosov actions
\textit{Pre-Print} (2019).


    
    
    
\bibitem{Anosov}
D.V. Anosov,
 Geodesic flows on compact manifolds with negative
curvature.
\textit{Trudy Mat. Inst. Steklov.}~\textbf{90} (1967), 3–210.


\bibitem{Apostolov}
V.  Apostolov, D. M. J. Calderbank, P. Gauduchon, E. Legendre,
Toric contact geometry in arbitrary codimension.\\
Preprint~ 2017 

\bibitem{Arbieto}
Arbieto, A; Morales, C.
\textit{Dynamics of Partial Actions}. Rio de Janeiro, RJ: IMPA, 2009.


\bibitem{Asao}
M. Asaoka,
On Invariant volumes of codimension-one Anosov flows and the Verjovsky conjecture.
\textit{Invent. Math.} \textbf{174} (2008), 435-462.


\bibitem{Bande}
G. Bande, A. Hadjar,
Contact Pairs,
\textit{Tohoku Math. J. (2)}~\textbf{57(2)}, (2005), 247--260.


\bibitem{Ba-Maq1}
T. Barbot, C.  Maquera, Transitivity of codimension one Anosov actions of $\R^k$ on closed manifolds.
\textit{Erg Theory and Dyn. Sys.} \textbf{31} (2011), 1-22.

\bibitem{Ba-Maq2}
T. Barbot, C.  Maquera,
On codimension one Anosov actions of $\R^k$ that are suspensions.
\textit{ Discrete and continuous Dynamical Systems - A} \textbf{29} (2011), no. 3, 803-822. 

\bibitem{Ba-Maq3}
T. Barbot, C.  Maquera,
 Algebraic Anosov actions of Nilpotent Lie groups.
  \textit{Topology and its Aplications}~\textbf{31} (2013), no. 1, 199-219. 


\bibitem{BFL}
Y. Benoist, P.  Foulon, F. Labourie,
Flots d'Anosov à distributions stable et instable différentiables.
\textit{J. Amer. Math. Soc.}~\textbf{5} (1992), 33-74.

\bibitem{Bolle}
P. Bolle, 
A contact condition for coisotropic manifolds of a symplectic manifold,
\textit{Comptes Rendus de l’Académie des Sciences, Série I} \textbf{01} (1996).

\bibitem{Fra}
J. Franks,
Anosov diffeomorphisms.
  \textit{Global Analysis (Proc. Symp. Pure Math.) Amer. Math. Soc.}~\textbf{14} (1970),
61-93.


\bibitem{ghys0}
E. Ghys,
Flots d'Anosov dont les feuilletages stables sont différentiables.
 \textit{Ann. Sci. Ecole Norm. Sup.} \textbf{4} (1987), no. 20, 251-270. 

\bibitem{ghys}
E. Ghys,
Codimension one Anosov flows and suspensions.
\textit{Lecture Notes in Math, Springer, Berlin,}~\textbf{1331} (1988), 59-72.


\bibitem{Ha-Th}
M. Handel, W.P. Thurston,
Anosov flows on new 3-manifolds. 
\textit{Invent. Math.} \textbf{59}  (1980), 95-103.

\bibitem{helgason}
S.Helgason,
	\textit{Differential Geometry, Lie Groups, and Symmetric Spaces}, Pure and Applied Mathematics (1979), Ed. Elsevier Science.



\bibitem{hirpu1}
M.Hirsch, C. Pugh,
Stable manifolds and hyperbolic sets.
 \textit{Proc. Sympos. Pure Math. - Amer. Math. Soc., Providence, RI,}~\textbf{13}  (1970), 133-164.

\bibitem{hirpush}
M.Hirsch, C. Pugh, M. Shub.
\textit{Invariant manifolds.
},  Lecture Notes in Math.,583, (1977) Springer,
Berlin .

\bibitem{HOF}
H.C. Hof,
An Anosov action on the bundle of Weyl chambers.
\textit{Ergod. Th. \& Dynam. Sys.} \textbf{5} (1985), 587-593.


\bibitem{kamm}
 Kammeyer, H
 \textit{An explicit rational structure for real semisimple Lie algebras}, Journal of Lie Theory, n2, vol 24, 307--319 (2014).

\bibitem{kal}
B. Kalinin, R. Spatzier,
On the Classification of Cartan Actions,
\textit{GAFA, Geom. funct. anal.}\textbf{17(2)} (2007),  468--490.

\bibitem{kat}
A. Katok, R.J. Spatzier,
First cohomology of Anosov actions of higher rank abelian groups and applications to rigidity.\textit{Publications Mathématiques de l'IHÉS}~\textbf{79} (1994),  131-156.

\bibitem{kobayashi}
Kobayashi, S.,  Nomizu, k.
\textit{Foundations of Differential Geometry. vol. 1},
Wiley Press, (1996).
	
\bibitem{Nesterov}
    Nesterov, A. I.\textit{Principal G-bundles}
 Nonassociative algebras and its applications: the fourth international conference, (2000),
  247--257.
	
\bibitem{New}
S.~E. Newhouse.
\textit{On codimension one Anosov diffeomorphisms.
},  Amer. J. Math., 92:761--770, 1970.


\bibitem{plante1}
J. Plante.
\textit{Anosov flows.
},  Amer. J. Math.  94:729--754, (1972)

\bibitem{Pla}
J. Plante.
\textit{Anosov flows, transversely affine foliations, and a conjecture of Verjovsky.
},  J. London Math. Soc. (2)23, no. 2, 359--362, (1981).



\bibitem{Pu-Shu}
C.~Pugh and M.~Shub.
\textit{Ergodicity of {A}nosov actions.
},  Invent. Math., 15:1--23, 1972.

\bibitem{Smale}
Smale, S.
\textit{Differentiable dynamical systems} Bull. Amer. Math. Soc. v. 73, p. 747--817, (1967).


\bibitem{To1}
P. Tomter.
\textit{Anosov flows on infra-homogeneous spaces.
},  Global Analysis (Proc. Sympos. Pure Math., Vol. XIV, Berkeley, Calif., 1968),
299--327, 1970.

\bibitem{To2}
P. Tomter.
\textit{On the classification of Anosov flows.
},  Topology, 14:179--189, 1975.

\bibitem{Van}
 Van Erp, E. \textit{Contact structures of arbitrary codimension and idempotents in the
	Heisenberg algebra}, (2010) pre-print.
	
\bibitem{Wis}
Wisser, F.
\textit{Classification of complex and real
semisimple Lie Algebras},
Dissertação Mestrado, Faculty of Science and Mathematics
the University of Vienna, (2001).
    
    

\end{thebibliography}


%
%
%
%
\end{document}